\documentclass[a4paper,10pt,leqno]{amsart}
\usepackage{entropy_style}

\begin{document}

\title{Simple Braids Tend toward Positive Entropy}

\author{Luke Robitaille}
\address{Massachusetts Institute of Technology, 77 Massachusetts Avenue, Cambridge, MA 02139}
\email{lrobitai@mit.edu}

\author[Minh-T\^{a}m Trinh]{Minh-T\^{a}m Quang Trinh}
\address{Massachusetts Institute of Technology, 77 Massachusetts Avenue, Cambridge, MA 02139}
\email{mqt@mit.edu}

\begin{abstract}
A simple braid is a positive braid that can be drawn so that any two strands cross at most once.
We prove that as $n \to \infty$, the proportion of simple braids on $n$ strands that have positive topological entropy tends toward $100\%$.
Notably, such braids are either pseudo-Anosov or reducible with a pseudo-Anosov component.
Our proof involves a method of reduction from simple braids to non-simple $3$-strand braids that may be of independent interest.
\end{abstract}

\maketitle

\thispagestyle{empty}



\section{Introduction}

\subsection{}


Let $\Br_n$ be the braid group on $n$ strands.
A braid $\beta \in \Br_n$ is \dfemph{simple} iff, in some planar diagram for $\beta$, the crossings are all positive and any two strands cross at most once.
The subset of simple braids $E_n \subseteq \Br_n$ forms a generating set: the most natural one for the Garside theory of $\Br_n$ \cite{em}.

At the same time, $\Br_n$ can be identified with the mapping class group of a disk with $n$ marked points rel its boundary.
By \cite{akm}, every self-map of a compact topological surface can be assigned a nonnegative real number called its \dfemph{topological entropy}, or \dfemph{entropy} for short, roughly measuring the growth rate of its mixing of open covers.
The entropy of a mapping class is defined to be the infimum of the entropies of the maps it represents.
The goal of this note is to prove:

\begin{thm}\label{thm:main}
The proportion of simple braids on $n$ strands that have positive topological entropy tends to 100\% as $n$ tends to infinity.
\end{thm}

In fact, we give a more precise version: \Cref{thm:main-quantitative}.
The idea of the proof is to reduce from studying simple braids on $n$ strands to studying non-simple braids on $3$ strands, whose positive entropy can be detected via the quotient homomorphism $\Br_3 \to \SL_2(\bb{Z})$.
The reduction step, as well as the combinatorics that ensures that sufficiently many of the resulting $3$-strand braids have positive entropy, may be of independent interest.

\subsection{}

One motivation for \Cref{thm:main} is the study of a different, but closely related, property of braids.
Recall that under the Nielsen--Thurston classification, a mapping class is either \dfemph{periodic}, \dfemph{reducible}, or \dfemph{pseudo-Anosov}.
These options amount to the possible dynamics for its action on simple closed curves \cite{thurston}.
The entropy and Nielsen--Thurston type of a mapping class constrain each other:
Namely, its entropy is zero if and only if it is either periodic or reducible with solely periodic components \cite{bb}.

Caruso and Wiest showed that if $n \geq 3$, then in the Cayley graph of $(\Br_n, E_n)$, the proportion of pseudo-Anosov braids in the ball of radius $\ell$ tends to 100\% as $\ell$ tends to infinity \cites{caruso, cw}.
This confirmed a folklore expectation dating to the work of Thurston.
Mahler and Sisto showed similar results, phrased in terms of random walks in non-elementary subgroups \cites{maher, sisto}.

The following sharpening of \Cref{thm:main}, which we leave to future work, would be directly complementary to the work of Caruso and Wiest.

\begin{conj}
The proportion of simple braids on $n$ strands that are pseudo-Anosov tends to 100\% as $n$ tends to infinity.
\end{conj}

\subsection{}

We thank the 2021 MIT PRIMES-USA program for putting us in contact, and for helping to fund this research.
We thank Stephen Bigelow, Benson Farb, and Reid Harris for answering questions about the Burau representation and topological entropy, and Tanya Khovanova and Kent Vashaw for proofreading an earlier draft.
During the PRIMES program, the second author was supported by an NSF Mathematical Sciences Research Fellowship, Award DMS-2002238.

We dedicate this note to the memory of Kevin James, who mentored the second author at the 2012 Clemson University REU in Computational Algebraic Geometry, Combinatorics, and Number Theory.

\section{Topological Entropy}\label{sec:entropy}

In this section, we collect the only properties of topological entropy that we actually need.

\subsection{}

Let $S$ be an compact topological surface, possibly with boundary, and $I \subset S$ a finite set of  points in its interior.
Let $M = \Mod(S, I, \partial S)$ be the mapping class group of $(S, I)$ rel the boundary $\partial S$.
Explicitly, $M = \pi_0(\Homeo^+(S, I, \partial S))$, where $\Homeo^+(S, I, \partial S)$ is the group of self-homeomorphisms of $S$ that stabilize $I$ and fix $\partial S$, endowed with the compact-open topology \cite[\S{2.1}]{fm}.

\subsection{}

We define the \dfemph{entropy} of a map $f : S \to S$ to be its topological entropy $h(f)$ in the sense of \cite{akm}.
We define the \dfemph{entropy} of a mapping  class $\phi \in M$ to be
\begin{align}
h(\phi) = \inf_{f \in \phi} h(f),
\end{align}
where the notation $f \in \phi$ means $f$ is a representative of $\phi$.

\begin{lem}\label{lem:akm}
The function $h : M \to \bb{R}_{\geq 0}$ has the following properties:
\begin{enumerate}
\item 	$h$ is constant along conjugacy classes.

\item 	For any $\phi \in M$ and integer $k > 0$, we have $h(\phi^k) \leq k h(\phi)$.

\end{enumerate}
\end{lem}

\begin{proof}
Parts (1) and (2) respectively follow from Theorems 1 and 2 in \emph{ibid.}
\end{proof}

\subsection{}

Suppose that $I' \subseteq I$.
By construction, a mapping class $\phi \in M$ can be lifted to $M' \vcentcolon= \Mod(S, I', \partial S)$ if and only if some, or equivalently any, representative of $\phi$ stabilizes $I'$.
We deduce that:

\begin{lem}\label{lem:stable}
If $\phi \in M$ lifts to $\phi' \in M'$, then $h(\phi) \geq h(\phi')$.
\end{lem}

\subsection{}

Let $D$ be a closed disk, and $I \subset D$ a finite set of points in its interior.
Let 
\begin{align}
\Br_I = \pi_1(\Conf^{|I|}(D), I),
\end{align}
where $\Conf^n(D)$ denotes the configuration space of $n$ unordered points in $D$.
As explained in \cite[\S{9}.1.3]{fm}, there is an explicit isomorphism
\begin{align}
\beta \mapsto \phi(\beta) : \Br_I \xrightarrow{\sim} \Mod(D, I, \partial D).
\end{align}
At the same time, we can identify $\Br_I$ with the usual braid group on $|I|$ strands, up to fixing an ordering of $I$.

We define the \dfemph{entropy} of a braid $\beta$ to be that of the corresponding mapping class:
$h(\beta) = h(\phi(\beta))$.
Now we can rewrite \Cref{lem:stable} in terms of braids.
For any $\beta \in \Br_I$ and $I' \subseteq I$, we say that $I'$ is \dfemph{stable} under $\beta$ iff we can delete strands from $\beta$ to obtain an element of $\Br_{I'}$.
In this case, we denote the new braid by $\beta|_{I'}$.
Since $\phi(\beta)$ lifts to $\phi(\beta|_{I'})$, \Cref{lem:stable} says that
\begin{align}\label{eq:stable}
h(\beta) \geq h(\beta|_{I'}).
\end{align}

\subsection{}
 
For any integer $N > 0$ and $g \in \Mat_N(\bb{C}[t^{\pm 1}])$, the characteristic polynomial of $g$ is a polynomial of degree $N$ with coefficients in $\bb{C}[t^{\pm 1}]$.
For any complex number $z \neq 0$, let $\Spec(g(z))$ be the eigenvalue spectrum of $g(z) \vcentcolon= g|_{t \to z}$, viewed as an unordered multiset of $N$ complex numbers.
The \dfemph{spectral radius} of $g$, which we will denote $\radius(g)$, is the maximum value of $|\lambda|$ as we run over $z$ on the unit circle and $\lambda \in \Spec(g(z))$.
The following result linking spectral radius to entropy was shown by Fried \cite{fried} and Kolev \cite{kolev} independently:

\begin{thm}[Fried, Kolev]\label{thm:radius}
Let $\rho : \Br_n \to \GL_n(\bb{Z}[t^{\pm 1}])$ be the unreduced Burau representation of $\Br_n$, as in \cite[\S{2}]{kolev}.
Then $\log \radius(\rho(\beta)) \leq h(\beta)$ for all $\beta \in \Br_n$.
\end{thm}

\begin{cor}
In \Cref{thm:radius}, the same conclusion would hold were $\rho$ the reduced, rather than unreduced, Burau representation.
\end{cor}

\begin{proof}
As a $\bb{Z}[t^{\pm 1}][\Br_n]$-module, the unreduced representation is the direct sum of the reduced representation and the trivial representation \cite[\S{1.3}]{turaev}.
\end{proof}

\begin{cor}\label{cor:sl-2}
Let $\gamma : \Br_3 \to \SL_2(\bb{Z})$ be the reduced Burau representation at $n = 3$ and $t = -1$.
Then $|{\tr(\gamma(\beta))}| > 2$ implies $h(\beta) > 0$ for all $\beta \in \Br_3$.
\end{cor}

\begin{proof}
If $|{\tr(\gamma(\beta))}| > 2$, then $\gamma(\beta)$ must have an eigenvalue greater than $1$. 
\end{proof}

\section{Simple Braids}\label{sec:braids}

\subsection{}

Following Artin, the braid group on $n$ strands has the presentation
\begin{align}
\Br_n = \left\langle \sigma_1, \ldots, \sigma_{n - 1}\, \middle|
	\begin{array}{ll}
	\sigma_i \sigma_{i + 1} \sigma_i = \sigma_{i + 1} \sigma_i \sigma_{i + 1}
		&i = 1, \ldots, n - 1,\\
	\sigma_i \sigma_j = \sigma_j \sigma_i
		&|i - j| > 1
	\end{array}\right\rangle,
\end{align}
where $\sigma_i$ represents the positive simple twist of the $i$th and $(i + 1)$th strands.
The \dfemph{writhe} of a braid on $n$ strands is its image under the quotient map $\ell : \Br_n \to \bb{Z}$ that sends $\ell(\sigma_i) = 1$ for all $i$.

\subsection{}

Let $S_n$ be the symmetric group on $n$ letters.
Let $s_i$ be the simple transposition that swaps $i$ and $i + 1$.
There is a quotient homomorphism $\Br_n \to S_n$ given by $\sigma_i \mapsto s_i$.

The set of simple braids $E_n \subseteq \Br_n$ is the image of a right inverse to this quotient map.
Indeed, every permutation $w \in S_n$ can be written as
\begin{align}
w = (s_{i_1} \cdots s_1) (s_{i_2} \cdots s_2) \cdots (s_{i_n} \cdots s_n)
\end{align}
for some uniquely determined $i_1, i_2, \ldots, i_n$ such that $j - 1 \leq i_j \leq n - 1$.
Let
\begin{align}
\sigma_w = (\sigma_{i_1} \cdots \sigma_1) (\sigma_{i_2} \cdots \sigma_2) \cdots (\sigma_{i_n} \cdots \sigma_n).
\end{align}
Then $w \mapsto \sigma_w$ is a right inverse of the quotient map $\Br_n \to S_n$, and furthermore, $E_n = \{\sigma_w \mid w \in S_n\}$ \cite{em}.

\subsection{}

For any $w \in S_3$ and integer $N > 0$, let
\begin{align}
P(w, N) = \{\vec{w} = (w_1, \ldots, w_N) \in S_3^N \mid w_1 \cdots w_N = w\}.
\end{align}
For any $\vec{w} \in S_3^N$, let $\sigma_{\vec{w}} = \sigma_{w_1} \cdots \sigma_{w_N}$.
We now show that many $3$-strand braids of the form $\sigma_{\vec{w}}$ for some $\vec{w} \in P(w, N)$ are braids of positive entropy.

\begin{lem}\label{lem:crazy}
Let $\beta_1, \ldots, \beta_k \in \Br_3$ be a list of $3$-strand braids such that:
\begin{enumerate}
\item 	They are all positive, \emph{i.e.}, can be written without negative powers of the $\sigma_i$.

\item 	They are all pure, \emph{i.e.}, map to the identity of $S_n$.

\item 	There is no matrix $g \in \SL_2(\bb{Z})$ such that
		\begin{align}
		|{\tr(g\gamma(\beta_i))}| \leq 2
			\quad\text{for all $i$}.
		\end{align}
		
\end{enumerate}
Let $L = \max_i \ell(\beta_i)$, the maximum writhe among the $\beta_i$.
Then
\begin{align}
|\{\vec{w} \in P(w, N) \mid h(\sigma_{\vec{w}}) > 0 \}| \geq 6^{N - L - 1}
\end{align}
for any integer $N \geq L + 1$.
\end{lem}

\begin{proof}
By \Cref{cor:sl-2}, it suffices to give a lower bound on the number of $\vec{w} \in P(w, N)$ such that $|{\tr(\gamma(\sigma_{\vec{w}}))}| > 2$.

We have complete freedom to pick the first $N - L - 1$ entries of $\vec{w}$.
We pick the $(N - L)$th entry to ensure that the product of the first $N - L$ entries of $\vec{w}$ equals $w$.
By condition (3), there must be some $i$ such that 
\begin{align}
|{\tr(\gamma(\sigma_{w_1} \cdots \sigma_{w_N - L}) \gamma(\beta_i))}| > 2.
\end{align}
Using condition (1), we can write $\beta_i = \sigma_{w_{N - L + 1}} \cdots \sigma_{w_{N - L + k}}$ for some $k \leq L$ and $w_{N - L + 1}, \ldots, w_{N - L + k} \in S_3$.
For $j$ such that $k < j \leq L$, we set $w_{N - L + j} = 1$.
Finally, by condition (2), the product of all of the entries in the resulting tuple $\vec{w}$ equals $w$.
\end{proof}

\begin{lem}\label{lem:6}
In the setup of \Cref{lem:crazy}, it is possible to choose the braids $\beta_i \in \Br_3$ so that $k = 4$ and $L = 6$.
Explicitly,
\begin{align}
(\beta_i)_{i = 1}^4
	= (1,\, \sigma_1^2 \sigma_2^2,\, \sigma_2^2 \sigma_1^2,\, \sigma_1^4 \sigma_2^2).
\end{align}
 
\end{lem}

\begin{proof}
Conditions (1)--(2) on the $\beta_i$ are immediate; it remains to check condition (3).
Without loss of generality, we can normalize the reduced Burau representation so that the homomorphism $\gamma$ in \Cref{cor:sl-2} takes the form
\begin{align}
\gamma(\sigma_1)
	= \pmat{1 &1 \\ 0 &1},
\quad\gamma(\sigma_2)
	= \pmat{1 &0 \\ -1 &1}.
\end{align}
We compute that
\begin{align}
(\gamma(\beta_i))_{i = 1}^4
	= \pa{
		\pmat{1&0\\ 0&1},
		\pmat{-3&2\\ -2&1},
		\pmat{1&2\\ -2&-3},
		\pmat{1&-2\\ 4&-7}
	}.
\end{align}
So we must show that there cannot exist $\psmat{a&b\\ c&d} \in \SL_2(\bb{Z})$ such that
\begin{align}
\label{eq:crazy1}
|a + d| &\leq 2,\\
\label{eq:crazy2}
|{-3}a - 2b + 2c + d|, |a - 2b + 2c - 3d| &\leq 2,\\
\label{eq:crazy4}
|a + 4b - 2c - 7d| &\leq 2.
\end{align}
(The structure of our argument will clarify why we group the inequalities in this way.)
In what follows, set $f = -2a + 2b$.
First, by \eqref{eq:crazy2},
\begin{align}
4|a - d| \leq |{-3}a + f + d| + |a + f - 3d| \leq 4,
\end{align}
from which we deduce
\begin{align}\label{eq:crazy5}
|a - d| \leq 1.
\end{align}
Next, by \eqref{eq:crazy1}, $2|a| \leq |a - d| + |a + d| \leq 3$, from which $a \in \{-1, 0, 1\}$.

If $a = 0$, then by \eqref{eq:crazy5}, $d \in \{-1, 0, 1\}$.
If $d = 0$, then \eqref{eq:crazy2} says $|f| \leq 2$, from which $-b + c \in \{-1, 0, 1\}$.
This contradicts the fact that $bc = ad - bc = 1$.
If $d = 1$, then \eqref{eq:crazy2} says $|f + 1|, |f - 3| \leq 2$.
This forces $f = 1$, contradicting the fact that $f \in 2\bb{Z}$.
The argument when $d = -1$ is similar, but with flipped signs.

If $a = 1$, then by \eqref{eq:crazy5}, $d \in \{0, 1\}$.
If $d = 0$, then the same argument as for $(a, d) = (0, 1)$ shows that $f = 1$, contradicting $f \in 2\bb{Z}$.
If $d = 1$, then \eqref{eq:crazy2} says $|f - 2| \leq 2$, from which $-b + c \in \{0, 1, 2\}$.
But also, $1 - bc = ad - bc = 1$, from which $bc = 0$.
Therefore, $(b, c) \in \{(0, 0), (0, 1), (0, 2), (-1, 0), (-2, 0)\}$, and each option contradicts \eqref{eq:crazy4}.

Finally, the argument when $a = -1$ is similar to that when $a = 1$, except that in the $d = -1$ subcase, the options for $(b, c)$ have flipped signs in both entries, so we can again conclude using \eqref{eq:crazy4}.
\end{proof}


\subsection{}

Next we explain how, using the sets $P(w, N)$ above, we can pass from simple braids on many strands to non-simple braids on $3$ strands that have equal or lower entropy.
For any integer $N > 0$, let
\begin{align}
C_N = \{w \in S_{3N} \mid \text{$w$ is a single cycle of length $3N$}\}.
\end{align}
We define a map
\begin{align}\label{eq:psi}
\vec{p} : C_N \to P(s_1s_2, N) \sqcup P(s_2s_1, N)
\end{align}
as follows.
First, for any $c \in C_N$ and residue class $i$ mod $3N$, let $a_i = a_i(c)$ be the image of $1$ under $c^i$, where we view $c^i$ as a permutation of $\{1, 2, \ldots, 3N\}$.
In other words, $c$ is the cycle $1 = a_0 \mapsto a_1 \mapsto \cdots \mapsto a_{3N} = 1$.
We define a permutation of $\{1, 2, 3\}$ in three stages:
\begin{enumerate}
\item 	A bijection $\{1, 2, 3\} \xrightarrow{\sim} \{a_{i - 1}, a_{i - 1 + N}, a_{i - 1 + 2N}\}$ sending $1$ to the smallest element of the target, $2$ to the next-smallest, and $3$ to the largest.

\item 	A bijection $\{a_{i - 1}, a_{i - 1 + N}, a_{i - 1 + 2N}\} \xrightarrow{\sim} \{a_i, a_{i + N}, a_{i + 2N}\}$ sending $a_k \mapsto a_{k + 1}$ for all $k$.

\item 	A bijection $\{a_i, a_{i + N}, a_{i + 2N}\} \xrightarrow{\sim} \{1, 2, 3\}$ sending the smallest element of the domain to $1$, the next-smallest to $2$, and the largest to $3$.

\end{enumerate}
For $i = 1, 2, \ldots, N$, let $p_i = p_i(c) \in S_3$ be the permutation of $\{1, 2, 3\}$ resulting from the construction above.

\begin{lem}\label{lem:fiber}
For all $c \in C_N$, the product $w_1(c) \cdots w_N(c)$ is a $3$-cycle, so \eqref{eq:psi} can be defined using $\vec{p}(c) \vcentcolon= (p_1(c), \ldots, p_N(c))$.
Moreover,
\begin{align}
|\{c \in C_N \mid \vec{p}(c) = \vec{w}\}| = \frac{|C_N|}{2 \cdot 6^{N - 1}}
\end{align}
for all $\vec{w} \in P(s_1s_2, N) \sqcup P(s_2s_1, N)$.
That is, the fibers of $\vec{p}$ are equinumerous.
\end{lem}

\begin{proof}
In the notation of the discussion above, the product $w = w_1(c) \cdots w_N(c)$ is the permutation of $\{1, 2, 3\}$ defined in stages by:
\begin{enumerate}
\item 	A bijection $\{1, 2, 3\} \xrightarrow{\sim} \{1, a_N, a_{2N}\}$ sending $1$ to $1$, $2$ to the next-smallest element, and $3$ to the largest.

\item 	A permutation of $\{1 = a_0, a_N, a_{2N}\}$ sending $a_k \mapsto a_{k + N}$ for all $k$.

\item 	A bijection $\{1, a_N, a_{2N}\} \xrightarrow{\sim} \{1, 2, 3\}$ sending $1$ to $1$, the next-smallest element to $2$, and the largest to $3$.

\end{enumerate}
We deduce that $w$ is the $3$-cycle that sends $1 \mapsto 2$, \emph{resp.}\ $1 \mapsto 3$, when $a_N < a_{2N}$, \emph{resp.}\ $a_{2N} < a_N$.
This proves the first assertion.

Next, observe that there are $\frac{|C_N|}{2 \cdot 6^{N - 1}}$ ways to form an ordered $N$-tuple $\vec{A} = (A_1, \ldots, A_N)$ of sets of size $3$ such that their union is $\{1, \ldots, 3N\}$ and $A_1 \ni 1$.
For any $\vec{w} = (w_1, \ldots, w_N) \in P(s_1s_2, N) \sqcup P(s_2s_1, N)$, we claim that there is an injective map from the set of such tuples into the set of cycles $c \in C_N$ for which $\vec{p}(c) = \vec{w}$.
Since the number of possible $\vec{w}$, \emph{resp.}\ $c$, is $2 \cdot 6^{N - 1}$, \emph{resp.}\ $|C_N|$, this will prove the second assertion.

It suffices to show that $\vec{w}$ determines a way of assigning the elements of $A_i$ bijectively to variables $a_{i - 1}, a_{i - 1 + N}, a_{i - 1 + 2N}$ for every $i$, such that $a_0 = 1$.
We use induction:
If $A_1 = \{1, a, b\}$, then the $3$-cycle $w_1 \cdots w_N$ determines whether we assign $(a_N, a_{2N})$ to be $(a, b)$ or $(b, a)$.
In general, once we have assigned the elements of $A_i$, the permutation $w_i$ determines how we assign the elements of $A_{i + 1}$.
\end{proof}

\begin{lem}\label{lem:cycle-entropy}
For all $c \in C_N$, we have 
\begin{align}
h(\sigma_c) \geq \frac{1}{N} h(\sigma_{\vec{p}(c)}).
\end{align}
\end{lem}

\begin{proof}
We use the setup and language of \Cref{sec:entropy}.
Let $I \subseteq D$ be a finite set of $3N$ points, and order them from $1$ to $3N$, so that we can identify $\Br_{3N}$ with $\Br_I$.
If $c$ is the $3N$-cycle $1 = a_0 \mapsto a_1 \mapsto \cdots \mapsto a_{3N} = 1$, then $c^N$ contains the $3$-cycle $1 \mapsto a_N \mapsto a_{2N} \mapsto 1$.
Thus $I' \vcentcolon= \{1, a_N, a_{2N}\}$ is stable under $\sigma_c^N$.
In fact, if we identify $\Br_{I'}$ with $\Br_3$ via some ordering, then $\sigma_c^N|_{I'}$ can be identified with $\sigma_{\vec{p}(c)}$ up to conjugacy.
Now, 
\begin{align}
N h(\sigma_c) \geq h(\sigma_c^N) \geq h(\sigma_c^N|_{I'}) = h(\sigma_{\vec{p}(c)})
\end{align}
by \Cref{lem:akm} and display \eqref{eq:stable}.
\end{proof}

\subsection{}

Combining \Crefrange{lem:crazy}{lem:cycle-entropy} gives:

\begin{lem}\label{lem:main}
For any $N \geq 7$, we have 
\begin{align}
\frac{|\{c \in C_N \mid h(\sigma_c) > 0\}|}{|C_N|}
	\geq 6^{-6}.
\end{align}
\end{lem}

For any $n$ and $w \in S_n$, we will call a cycle of $w$ \dfemph{relevant} if its length is divisible by $3$ and at least $3 \cdot 7 = 21$, and \dfemph{irrelevant} otherwise.
We apply the same name to the corresponding orbit, \emph{i.e.}, to the underlying unordered subset of $\{1, \ldots, n\}$.
We define an equivalence relation on $S_n$ as follows:
$w \approx w'$ means that $w$ and $w'$ have the same irrelevant cycles and the same relevant orbits.

We arrive at the following result, reducing the proof of \Cref{thm:main} to exhibiting sufficiently many elements of $S_n$ with sufficiently many relevant cycles.

\begin{prop}\label{prop:main}
Let $D \subseteq S_n$ be an equivalence class for the relation $\approx$ in which the elements each have $r$ relevant cycles.
Then 
\begin{align}
\frac{|\{w \in D \mid h(\sigma_w) > 0\}|}{|D|} \geq 1 - (1 - 6^{-6})^r.
\end{align}
\end{prop}

\begin{proof}
Let $\cal{O}$ be the collection of relevant orbits arising from elements of $D$.
By restricting any element of $D$ to its behavior on these orbits, we get a bijection
\begin{align}
D \xrightarrow{\sim} \{(c_O)_{O \in \cal{O}} \mid \text{$c_O$ is an $|O|$-cycle with underlying orbit $O$}\}.
\end{align}
Moreover, if $w \mapsto (c_O)_O$, then $h(\sigma_w) \geq \max_O h(\sigma_{c_O(w)})$ by \eqref{eq:stable}.
So it remains to bound the proportion of tuples $(c_O)_O$ that have $h(\sigma_{c_O}) = 0$ for all $O$.

For any $O \in \cal{O}$, we must have $|O| = 3N$ for some $N \geq 7$.
Fix an ordering of the elements of $O$, so that we can identify the possibilities for $c_O$ with elements $c \in C_N$.
Then, by \Cref{lem:akm}(1) and \Cref{lem:main}, the proportion of possibilities for $c_O$ with $h(\sigma_{c_O}) = 0$ is at most $1 - 6^{-6}$.
Applying this argument to each of the $r$ relevant orbits, we see that the proportion of tuples $(c_O)_O$ that have $h(\sigma_{c_O}) = 0$ for all $O$ is at most $(1 - 6^{-6})^r$.
\end{proof}

\section{Permutations with Many Long Cycles}\label{sec:cycles}

\subsection{}

For any integers $n, \ell, r > 0$, let
\begin{align}
S_n(\ell, r) = \left\{w \in S_n\, \middle| 
	\begin{array}{l}
		\text{$w$ has at least $r$ cycles of length divisible by $3$}\\
		\text{and length at least $3\ell$}
	\end{array}
	\right\}.
\end{align}
The goal of this section is to prove:

\begin{prop}\label{prop:cycles}
For fixed $\ell, r > 0$ and $n \gg_{\ell, r} 0$, we have
\begin{align}
\frac{|S_n(\ell, r)|}{|S_n|}
	= 1 - o\pa{r\pa{\frac{n}{3\ell \cdot 2^r}}^{-\frac{1}{6\ell \cdot 2^r}}},
\end{align}
where the little-o constant is independent of $\ell, r$.
\end{prop}

\begin{proof}
By \Cref{lem:avoid} below, we know that the proportion of elements of $S_n$ that have no cycles of length $j \cdot 3 \cdot 2^i$ with $j$ odd is $O((n/(3 \cdot 2^i))^{-\frac{1}{6 \cdot 2^i}})$.

Let $i_0 = \lceil \log_2(\ell)\rceil$.
Then the proportion of elements that have at least one cycle of length $j \cdot 3 \cdot 2^i$ with $j$ odd, for \emph{each} $i$ such that $i_0 + 1 \leq i < i_0 + r$, is 
\begin{align}
1 - o\pa{r\pa{\frac{n}{3\ell \cdot 2^r}}^{-\frac{1}{6\ell \cdot 2^r}}}
	\quad\text{for $n \gg_{\ell, r} 0$}.
\end{align}
In any such element, these $r$ cycles must be pairwise distinct because their lengths are.
Moreover, their lengths are divisible by $3$ and at least $3\ell$.
\end{proof}

\subsection{}

For any integers $n, k > 0$, let
\begin{align}
X_{n, k} = \{w \in S_n \mid \text{$w$ has no cycles of length $k, 3k, 5k, \ldots$}\}.
\end{align}

\begin{lem}\label{lem:stanley}
We have
\begin{align}\label{eq:stanley}
\sum_{n \geq 0}
	{|X_{n, k}|} \frac{x^n}{n!}
	=
	(1 - x)^{-1} (1 - x^k)^{\frac{1}{k}}
	(1 - x^{2k})^{-\frac{1}{2k}}.
\end{align}
\end{lem}

\begin{proof}
For each integer $m > 0$, fix an indeterminate $t_m$, and for each $w \in S_n$, let $\lambda_m(w)$ be the number of $m$-cycles in $w$.
Display (5.30) of \cite{stanley} says that
\begin{align}
\sum_{n \geq 0} \sum_{w \in S_n}
	t_1^{\lambda_1(w)} \cdots t_n^{\lambda_n(w)}
	\frac{x^n}{n!}
	=
	\exp\pa{
		\sum_{m \geq 1}
		t_m
		\frac{x^m}{m}
	},
\end{align}
where $\exp(X) = \sum_{n \geq 0} \frac{X^n}{n!}$ as a formal series.
Now set $t_m = 0$ whenever $m = jk$ with $j$ odd and $t_m = 1$ for all other $m$.
The left-hand side simplifies to that of \eqref{eq:stanley}, while the right-hand side simplifies to
\begin{align}
\exp\pa{
		\sum_{m \geq 1} \frac{x^m}{m}
		-
		\sum_{j \geq 1} \frac{x^{jk}}{jk}
		+
		\sum_{i \geq 1} \frac{x^{2ik}}{2ik}
	}.
\end{align}
To finish, use $\exp(\sum_{m \geq 1} \frac{X^m}{m}) = (1 - X)^{-1}$.
\end{proof}

\begin{lem}\label{lem:avoid}
For fixed $k > 0$ and $n \gg_k 0$, we have
\begin{align}
\frac{|X_{n, k}|}{|S_n|} = O\pa{\pa{\frac{n}{k}}^{-\frac{1}{2k}}},
\end{align}
where the big-O constant is independent of $k$.
\end{lem}

\begin{proof}
We will study the right-hand side of \eqref{eq:stanley}.
First, we expand 
\begin{align}
\label{eq:binom-1}
(1 - x)^{-1} (1 - x^k)^{\frac{1}{k}} 
	&= \frac{1 - x^k}{1 - x}
			\sum_{i \geq 0}
			\binom{\frac{1}{k} - 1}{i}
			(-1)^i
			x^{ik},\\
\label{eq:binom-2}
(1 - x^{2k})^{-\frac{1}{2k}}
	&= \sum_{i \geq 0}
		\binom{-\frac{1}{2k}}{i}
		(-1)^i
		x^{2ik}
\end{align}
The right-hand side of \eqref{eq:binom-1} simplifies to a power series with nonnegative coefficients.
As for \eqref{eq:binom-2}:
Observe that for any $\alpha \in (0, 1]$ and integer $i \geq 0$, we have
\begin{align}
\binom{-\alpha}{i}(-1)^i = \frac{\alpha(\alpha + 1) \cdots (\alpha + i - 1)}{i!} \leq 2 \cdot \frac{\alpha(\alpha + 1) \cdots (\alpha + 2i - 1)}{(2i)!} = 2\binom{-\alpha}{2i},
\end{align}
where we handle $i = 0$ separately to prove the inequality.
Therefore, for each $i \geq 0$, the coefficient of $x^{2ik}$ on the right-hand side of \eqref{eq:binom-2} is nonnegative and bounded above by $2\binom{-1/(2k)}{2i}$.
That is, the coefficients in the series expansion of $(1 - x^{2k})^{-\frac{1}{2k}}$ are nonnegative and bounded above by the respective coefficients in the series expansion of $2(1 - x^k)^{-\frac{1}{2k}}$.

Altogether, by \Cref{lem:stanley}, $|X_{n, k}|/|S_n|$ is bounded above by the coefficient of $x^n$ in the series expansion
\begin{align}
2(1 - x)^{-1} (1 - x^k)^{\frac{1}{k}} (1 - x^k)^{-\frac{1}{2k}}
&=
	2(1 - x)^{-1} (1 - x^k)^{\frac{1}{2k}} \\
&=	
	2\pa{\frac{1 - x^k}{1 - x}}
	\sum_{i \geq 0}
		\binom{\frac{1}{2k} - 1}{i}(-1)^i x^{ik}\\
&=
	2\sum_{n \geq 0}
		\binom{\frac{1}{2k} - 1}{\lfloor \frac{n}{k}\rfloor}(-1)^{\lfloor \frac{n}{k} \rfloor} x^n.
\end{align}
Finally, for any $\alpha \in \bb{R} \setminus \bb{Z}_{\geq 0}$, it is known \cite[Thm.\ 2]{levrie} that
\begin{align}
\left|\binom{\alpha}{m}\right|
	\sim
		\frac{1}{|\Gamma(-\alpha)m^{1 + \alpha}|}
		\quad\text{as $m \to \infty$}.
\end{align}
Note that taking $\alpha = \frac{1}{2k} - 1$ gives $\frac{1}{2} \leq -\alpha \leq 1$.
On this interval, $\Gamma(-\alpha) \geq 1$, so we're done.
\end{proof}

\section{Conclusion}

\noindent
What follows is a quantitative refinement of \Cref{thm:main}.

\begin{thm}\label{thm:main-quantitative}
For any $0 < \epsilon < 1$ and integer $r > \log_{1 - 6^{-6}}(\epsilon)$, we can pick $N$ large enough that for all $n \geq N$, we have
\begin{align}\label{eq:quantitative}
\frac{|S_n(7, r)|}{|S_n|} \cdot (1 - (1 - 6^{-6})^r)
	&> 1 - \epsilon
\end{align}
in the notation of \Cref{sec:cycles}.
For such $n$, the proportion of simple braids on $n$ strands that have positive topological entropy is greater than $1 - \epsilon$.
\end{thm}

\begin{proof}
\Cref{prop:cycles} implies the first claim.
\Cref{prop:main} gives
\begin{align}
\frac{|\{w \in S_n(7, r) \mid h(\sigma_w) > 0\}|}{|S_n(7, r)|} \geq 1 - (1 - 6^{-6})^r,
\end{align}
from which
\begin{align}
\frac{|\{w \in S_n \mid h(\sigma_w) > 0\}|}{|S_n|}
	&\geq \frac{|S_n(7, r)|}{|S_n|} \cdot \frac{|\{w \in S_n(7, r) \mid h(\sigma_w) > 0\}|}{|S_n(7, r)|}\\
	&> 1 - \epsilon,
\end{align}
proving the second claim.
\end{proof}





\end{document}